\documentclass[11pt,a4paper,reqno]{amsart}
\usepackage{amsmath}
\usepackage{amsfonts}
\usepackage{amssymb}
\usepackage{amscd}
\usepackage{url}
\usepackage{enumerate}
\usepackage[pdftex,bookmarks=true]{hyperref}
\usepackage{xcolor}
\numberwithin{equation}{section}

\renewcommand\Re{{\operatorname{Re}}}
\renewcommand\Im{{\operatorname{Im}}}

\newcommand\R{{\mathbf{R}}}
\newcommand\C{{\mathbf{C}}}

\renewcommand\P{{\mathbf{P}}}
\newcommand\E{{\mathbf{E}}}

\newcommand\T{{\mathbf{T}}}

\newcommand\Var{{\operatorname{Var}}}
\newcommand\Cov{{\operatorname{Cov}}}
\newcommand\Area{{\operatorname{Area}}}

\newcommand\dist{{\operatorname{dist}}}
\newcommand\Z{{\mathbf{Z}}}

\newcommand\eps{\varepsilon}

\newcommand\Bg{{\mathbf g}}

\newcommand\CA{{\mathcal A}}
\newcommand\CB{{\mathcal B}}
\newcommand\CC{{\mathcal C}}
\renewcommand\CD{{\mathcal D}}
\newcommand\CE{{\mathcal E}}
\newcommand\CF{{\mathcal F}}

\newcommand\CM{{\mathcal M}}

\newcommand\CT{{\mathcal T}}

\newcommand\CZ{{\mathcal Z}}

\newcommand{\norm}[1]{\left\lVert #1 \right\rVert}
\newcommand{\ab}[1]{\left| #1 \right|}

\renewcommand\mod{\ \operatorname{mod}\ }

\newcommand\ep{\varepsilon}

\newcommand\height{\mathbf {ht}}

\theoremstyle{plain}
  \newtheorem{theorem}{Theorem}[section]	
  \newtheorem{conjecture}[theorem]{Conjecture}
  
  \newtheorem{assumption}[theorem]{Assumption}

  \newtheorem{fact}[theorem]{Fact}
  \newtheorem{lemma}[theorem]{Lemma}
  
  \newtheorem{question}[theorem]{Question}
  \newtheorem{example}[theorem]{Example}
  \newtheorem{condition}{Condition}
  \newtheorem{remark}[theorem]{Remark}
  
  \newtheorem{claim}[theorem]{Claim}

\theoremstyle{definition}

\begin{document}

\title[Random eigenfunctions on flat tori]{Random eigenfunctions on flat tori: universality for the number of intersections}
\author{Mei-Chu Chang, Hoi Nguyen, Oanh Nguyen, Van Vu}

\address{Department of Mathematics, University of California,
  Riverside, CA 92521}
\email{mcc@math.ucr.edu}

\address{Department of Mathematics, The Ohio State University, Columbus, Ohio 43210}
\email{nguyen.1261@math.osu.edu}

\address{Department of Mathematics, Yale University, New Haven, CT 06520}
\email{oanh.nguyen@yale.edu; van.vu@yale.edu}


\thanks{M.~C.~Chang is partially supported by NSF grant~DMS~1600154, the author thanks the Mathematics Department of 
University of California at Berkeley for its hospitality. H.~Nguyen is
partially supported by NSF grant~DMS~1600782. O.~Nguyen and V.~Vu are
partially  supported by NSF grant DMS-1307797 and AFORS grant FA9550-12-1-0083.}

\maketitle
\begin{abstract}
We show that several statistics of the number of intersections between
random eigenfunctions of general eigenvalues with a given smooth curve in
flat tori are universal under various families of randomness. 
\end{abstract}

\tableofcontents
\section{Introduction}

Let $\CM$ be a smooth Riemannian manifold. Let $F$ be a real-valued eigenfunction of the Laplacian on $\CM$ with eigenvalues $\lambda^2$,
$$-\Delta F = \lambda^2 F.$$
The nodal set $N_F$ is defined to be
$$N_F :=\{x\in \CM, F(x)=0\}.$$
The study of $N_F$ is extremely important in analysis and differential geometry. In this note we are simply interested in the case when $\CM$ is the flat tori $\T^d =\R^d/\Z^d$ with  $d\ge 2$; more specifically we will be focusing on the intersection set of $N_F$ with a given reference curve.

Let $\CC\subset \CM$ be a curve assumed to have unit  length with the arc-length parametrization $\gamma: [0,1] \to \CM$. The nodal intersection between $F$ and $\CC$ is defined as
$$\CZ(F) := \# \{x : F(x) = 0\} \cap \CC.$$
\subsection{Deterministic results in $\T^2$} It is known that all eigenvalues $\lambda^2$ have the form $4\pi^2 m, m\in \Z^+$. Let $\CE_\lambda$ be the collection of $\mu=(\mu_1,\mu_2) \in \Z^2$ such that
$$\mu_1^2 + \mu_2^2 = m.$$
Denote $N=N_m = \# \CE_\lambda$, that is $N=r_2(m)$. Note that  in this case, if $m=m_1^2 m_2$ with $m_1=2^r \prod_{q_k \equiv 3\mod 4} q_k^{b_k}$ and $m_2= 2^c \prod_{p_j \equiv 1 \mod 4} p_j^{a_j}$ ($c=0,1$) then (see, for example, \cite{S})
$$N = \prod_j (a_j+1).$$

The toral eigenfunctions $f(x) =e^{2\pi i \langle \mu, x\rangle}, \mu \in \CE_\lambda$ form an orthonormal basis in the eigenspace corresponding to $\lambda^2$.  We first introduce several deterministic results by Bourgain and Rudnick from \cite{BR0,BR1,BR2}. 

\begin{theorem}\label{theorem:det:upper}
Let $\CC\subset \T^2$ be a real analytic curve with nowhere vanishing curvature, then
$$\CZ(F) \le c \lambda.$$
\end{theorem}

The constant $c$ depends on the curve $\CC$.  This bound can be achieved from \cite{TZ} once we have 
$$\int_\CC |F|^2 d \gamma \gg e^{-c\lambda} \int_{\T^2} |F(x)|^2 dx.$$
This type of restriction result was obtained in \cite{BR1} in the stronger form
\begin{equation}\label{eqn:restriction}
\int_\CC |F|^2 d \gamma \gg  \int_{\CM} |F(x)|^2 dx.
\end{equation}

Henceforth the bound of Theorem \ref{theorem:det:upper} follows immediately. 

The lower bound for $\CZ(F)$ is also of special interest. Let $B_\lambda$ denote the maximal number of lattice points which lie on an arc of size $\sqrt{\lambda}$ on the circle $|x|=\lambda$
$$B_\lambda = \max_{|x| =\lambda} \# \{\mu \in \CE: |x-\mu| \le \sqrt{\lambda}\}.$$

\begin{theorem}\cite{BR2}\label{theorem:BR}
If $\CC \subset \T^2$ is smooth with nowhere vanishing curvature, then
$$\CZ(F) \gg \frac{\lambda}{B_\lambda^{5/2}}.$$
\end{theorem}

In particularly, as one can show that $B_\lambda\ll \log \lambda$ (see \cite{BR2}), we have
\begin{theorem}
$$\CZ(F)\gg \lambda^{1-o(1)}.$$
\end{theorem}
According to a conjecture of \cite{CG}, $B_\lambda = O(1)$ uniformly. This is known to hold for almost all $\lambda^2$, see for instance \cite[Lemma 5]{BR}; we also refer the reader to Lemma \ref{lemma:separation:d=2} of Section \ref{section:smallinterval} for a similar result (with a relatively short proof). In view of Theorem \ref{theorem:BR},  the following was conjectured in \cite{BR2}

\begin{conjecture}\label{conj:BR}
If $\CC \subset \T^2$ is smooth with non-zero curvature, then
$$\CZ(F) \gg \lambda .$$
\end{conjecture}

\subsection{Arithmetic random wave model} We next introduce a probabilistic setting first studied by Rudnick and Wigman \cite{RW}. Consider the random gaussian function
$$F(t) =\frac{1}{\sqrt N} \sum_{\mu\in \CE_\lambda} \eps_\mu e^{2\pi i \langle \mu, \gamma(t)\rangle },$$
 where $\eps_\mu$ are iid complex standard gaussian  with a saving
 $$\eps_{-\mu} = \bar{\eps}_\mu.$$
The random function $F$ is called {\it arithmetic random wave} \cite{Berry,KKW}, whose distribution is invariant under rotation by the gaussian property of the coefficients.

We now introduce the main result of \cite{RW}.

\begin{theorem}\label{theorem:gaussian:d=2:expectation}
Let $\CC \subset \T^2$ be a smooth curve on the torus, with nowhere vanishing curvature and of total length one. Then
\begin{enumerate}
\item the expected number of nodal intersections is precisely
$$\E_{\Bg} \CZ = \sqrt{2m},$$
\item the variance is bounded
$$\Var_{\Bg}(\CZ) \ll \frac{m}{N}.$$
\item Furthermore, let $\{m\}$ be a sequence such that $N_m\to \infty$ and $\{\hat{\tau}_m(4)\}$ do not accumulate at $\pm 1$, then
$$\Var_\Bg(\CZ) =  \frac{m}{N}\int_\CC \int_\CC4\left ( \frac{1}{N}  \left \langle \frac{\mu}{|\mu|}, \dot{\gamma}(t_1) \right \rangle^2  \left \langle  \frac{\mu}{|\mu|}, \dot{\gamma}(t_2) \right \rangle^2 -1\right )dt_1dt_2 + O(\frac{m}{N^{3/2}}).$$
\end{enumerate}
\end{theorem}

Here the subscript $\Bg$ is used to emphasize standard gaussian randomness, and $\tau_m$ is the probability measure on the unit circle $S^1 \subset \R^2$ associated to $\CE_\lambda$,
$$\tau_{m} = \frac{1}{N} \sum_{\mu\in \CE} \delta_{\mu/\sqrt{m}}.$$ 

A simple  consequence of (1) and (2) is that Conjecture
\eqref{conj:BR} holds for the random wave $F$ asymptotically almost
surely. In fact, the statement of (2) and (3) show that the variance
is much smaller than $m$, indicating a large number of cancellations
in the formula of the variance.

\subsection{Partial results in $\T^3$} Bourgain and Rudnick
\cite{BR0,BR1,BR2} also considered the intersection $\CZ$ between $N$
and a smooth {\it hypersurface} $\sigma$ for general $\T^d$.  For
$\T^3$, they obtained an analog  of Theorem \ref{theorem:det:upper}
for the $L^2$ restriction over $\CZ$. 
However, we are not aware of similar
deterministic results regarding the intersection with a smooth curve as
in $\T^2$. On the probabilistic side, Rudnick, Wigman and Yesha
\cite{RWY} have recently extended Theorem \ref{theorem:gaussian:d=2:expectation} to $\T^3$. Here, for $\lambda^2=4\pi^2 m$ with $m \neq 0,4,7 \mod 8$, let $\CE_\lambda$ be the collection of $\mu=(\mu_1,\mu_2,\mu_3) \in \Z^3$ such that $\mu_1^2 + \mu_2^2 + \mu_3^2= m$. Again denote $N=N_m= \# \CE_\lambda$. 

Consider the random gaussian function
$$F(t) =\frac{1}{\sqrt N} \sum_{\mu\in \CE_\lambda} \eps_\mu e^{2\pi i \langle \mu, \gamma(t)\rangle },$$
where $\eps_\mu$ are iid complex standard gaussian again with the saving 
$$\eps_{-\mu} = \bar{\eps}_\mu.$$ 
Rudnick, Wigman and Yesha showed the following result.
\begin{theorem}\label{theorem:gaussian:d=3}
Let $\CC \subset \T^3$ be a smooth curve on the torus of total length one with nowhere zero curvature. Assume further that either $\CC$ has nowhere-vanishing torsion or $\CC$ is planar. Then
\begin{itemize}
\item The expected number of nodal intersections is precisely
$$\E_{\Bg} \CZ = \frac{2}{\sqrt{3}}\sqrt{m} .$$
\item There exists $c>0$ such that
$$\Var_{\Bg}(\CZ) \ll \frac{m}{N^c}.$$
\end{itemize}
\end{theorem}

The proof of Theorem \ref{theorem:gaussian:d=2:expectation} and Theorem \ref{theorem:gaussian:d=3} are based on Kac-Rice formula. Let us sketch the computation of expectation for $d\ge 2$ that 
\begin{equation}\label{eqn:gaussian:expectation:d}
\E_{\Bg} \CZ = \frac{2}{\sqrt{d}} \sqrt{m}.
\end{equation}
We follow the proof of \cite[Lemma 2.3]{RWY}. Let $r(t_1,t_2) = \E (F(t_1) F(t_2))$. Denote $K_1(t)$ be the gaussian expectation (first intensity)
$$K_1(t) := \frac{1}{\sqrt{2\pi}} \E (|F'(t)| \big| F(t)=0).$$
By the Kac-Rice formula
$$\E \CZ = \int_{0}^1 K_1(t)dt.$$
Let $\Gamma$ be the covariance matrix of $(F(t),F'(t))$,
$$\Gamma(t) = \left ( \begin{matrix} r(t,t) &  r_1(t,t) \\ r_2(t,t) &
  r_{12}(t,t) \end{matrix} \right ),$$
where $r_1=\partial r/\partial t_1, r_2=\partial r/\partial t_2,
r_{12} =\partial^2 r/\partial t_1 \partial t_2$. 
It is not hard to show that  $\Gamma(t) = \left (\begin{matrix} 1 & 0 \\  0 & \alpha \end{matrix}\right )$, where $\alpha = r_{12}(t,t) = \frac{4}{d} \pi^2 m$. It thus follows 
$$K_1(t) = \frac{1}{\pi} \sqrt{\alpha} = \frac{2}{\sqrt{d}} \sqrt{m}.$$
For the variance, denote $K_2(t)$ to be
$$K_2(t) := \phi_{t_1,t_2}(0,0) \E \Big(|F'(t_1) F'(t_1)| \big| F(t_1)=0,F(t_2)=0\Big),$$
where $\phi_{t_1,t_2}$ is the density function of the random gaussian vector $(F(t_1),F(t_2))$. It is known that if the covariance matrix $\Sigma(t_1,t_2)$ of the
vectors $(F(t_1), F(t_2), F'(t_1), F'(t_2))$ is non-singular for all
$(t_1,t_2) \in A \times B$, then
$$\E (\CZ \restriction_A \CZ\restriction_B)-\E (\CZ\restriction_A) \E (\CZ\restriction_B)  = \int_{A \times B} K_2(t_1,t_2) dt_1 dt_2.$$

The main problem here is that the matrix $\Sigma(t_1,t_2)$ is not
always non-singular in $[0,1]^2$. Roughly speaking, to overcome this
highly technical obstacle, Rudnick and Wigman \cite{RW} and Rudnick,
Wigman and Yesha \cite{RWY} divide $[0,1]$ into subintervals $I_i$ of
length of order $1/\sqrt{m}$ each, and then show that Kac-Rice's
formula is available locally on most of the cells $I_i \times I_j$. We
refer the reader to \cite{RW, RWY} for more detailed treatment of
these issues.

\subsection{More general random waves and our main results}\label{subsection:results} Motivated by Conjecture \ref{conj:BR}, and by the universality phenomenon in probability, we are interested in the behavior of $\CZ(F)$ for other random eigenfunctions $F$ beside the gaussian arithmetic random waves as above. More specifically, consider the random function
\begin{equation}\label{eqn:F}F(t) = \frac{1}{\sqrt{N}}\sum_{\mu\in \CE_\lambda} \eps_\mu e^{2\pi i \langle \mu,
  \gamma(t)\rangle},
\end{equation}
where $\eps_\mu =\eps_{1,\mu} + i \eps_{2,\mu}$, where
$\eps_{1,\mu},\eps_{2,\mu}, \mu \in \CE_\lambda$ are iid random variables with the saving constraint $\eps_{-\mu}
= \bar{\eps}_{\mu}$ so that $F(t)$ is real-valued as in the gaussian case.

We denote by $\P_{\ep_\mu}, \E_{\ep_\mu}$, and $\Var_{\ep_\mu}$ the probability, expectation, and variance with respect to the random variables $(\ep_\mu)_{\mu\in \CE_{\lambda}}$.

We are interested in the following problem.

\begin{question}\label{conj:expectation}  Are the
 statistics such as $\E_{\ep_\mu} \CZ(F)$ and $\Var_{\ep_\mu}(\CZ(F))$ with respect to the randomness of the random variables $\ep_\mu$ universal? 
\end{question}
Note that we can write $F(t)$ as
 \begin{equation}
 F(t) = \frac{1}{\sqrt{N}}\sum_{\mu\in \CE_\lambda} \eps_\mu e^{2\pi i \langle \mu, \gamma(t)\rangle} = \frac{1}{\sqrt{N}}\sum_{\mu\in \CE_\lambda} \eps_{1,\mu} \cos (2\pi \langle \mu, \gamma(t)\rangle) +  \eps_{2,\mu} \sin (2\pi \langle \mu, \gamma(t)\rangle). \label{def-F}
 \end{equation}

We now restrict to $\T^2$ by assuming several necessary properties
of the curves and distributions.

{\bf Assumption on the reference curve.} Let $\gamma(t), t\in \in [0,1]$ be a curve of unit length. 
\begin{condition} \label{cond-curve} We will suppose the following
\begin{enumerate}[(i)]

\item\label{cond-ana} (Analyticity) The function $\gamma(t)$ extends analytically to $t\in [0,
  1]\times [-\ep, \ep]$ for some small constant $\ep$.
  
\item (Non-vanishing curvature)\label{cond-nondegeneracy} The curve $\gamma(t): [0,1] \to \T^2$  has arc-length parametrization with positive curvature. More specifically, there exists a positive constant $c$
such that $\|\gamma'(t) \|=1$ and $\|\gamma''(t)\| >c$ for all
$t$. 
\end{enumerate}
\end{condition}

We remark that these conditions imply that for any constant $c_0>0$, there exists a constant
  $\alpha>0$ such that for any interval $I \subset [0,1]$ of length
  $c_0/\lambda$, the segment $\{\gamma(t), t\in I\}$ cannot be contained in a ball
  of radius $N^{-\alpha}/\lambda$.

{\bf Assumption on the distribution.} We will assume $\ep_\mu$ to have mean zero, variance one with the following properties.
\begin{condition}\label{cond-variable} There is a fixed number $K$ such that either

\begin{enumerate}[(i)]
\item (Continuous distribution)\label{cond-cont} $\ep_\mu$ is absolutely continuous with density function $p$ bounded $\|p\|_\infty \le K$.
\vskip .1in
\item (Mixed distribution) \label{cond-discrete} There exist positive constants $c_1,c_2,c_3$ such that $\P(c_1\le |\ep_\mu - \ep_\mu'|\le c_2)\ge c_3$ where $\ep_\mu'$ is an independent copy of $\eps_\mu$ and one of the following holds
\vskip .05in
\begin{itemize}
\item  either $|\ep_\mu|> 1/K$ with probability one 
\vskip .05in
\item or $\ep_\mu 1_{|\ep_\mu|\le 1/K}$ is continuous with density bounded above by $K$.  
\end{itemize}
\end{enumerate}

\end{condition}

The assumption that $\ep_\mu$ stays away from zero (for discrete distribution) is necessary because
otherwise the random function $F(t)$ might be vanishing with positive
probability. One representative example of our consideration is
Bernoulli random variable which takes
values $\pm 1$ with probability 1/2.  We now state our main result for $\T^2$.

\begin{theorem}[general distributions in
  $\T^2$]\label{theorem:general} With $\gamma$ as above, assume that $\eps_{1,\mu},
 \eps_{2,\mu}, \mu \in \CE_\lambda$ are iid random variables satisfying Condition \ref{cond-variable}. Then  for almost all $m$ we have

%
%
\begin{itemize}
\item $\E_{\ep_\mu} \CZ = \E_{\Bg} \CZ + O(\lambda/N^{c})$;
\vskip .05in
\item more generally, for any fixed $k$, $\E_{\ep_\mu} \CZ^k = \E_{\Bg} \CZ^k + O(\lambda^k/N^{c})$,
\end{itemize}
where the subscript ${\Bg}$ stands for the distribution in which the $\ep_{1, \mu}$ and $\ep_{2, \mu}$ are independent standard gaussian.
Here the implicit constants depend on the curve $\gamma$ and $k$ but not on $N$ and $\lambda$. In particularly, with $\gamma$ and $\lambda$ as in Theorem \ref{theorem:gaussian:d=2:expectation} 
$$\E_{\ep_\mu} \CZ =  \sqrt{2m} + O(\lambda/N^{c}) \mbox{ and } \Var_{\ep_\mu}(\CZ) \ll  \frac{\lambda^2}{N^c}.$$
\end{theorem}

The density of the sequence $\{m\}$ above can be worked out explicitly, but we will omit the details. It is plausible to conjecture that the variance is indeed as small as in (iii) of Theorem \ref{theorem:gaussian:d=2:expectation}. However, this is an extremely delicate matter given the highly nontrivial analysis of the gaussian case. 

To prove Theorem \ref{theorem:general}, we will need to show that the set $\CE_{\lambda}$ satisfies the following assumption which is later proven to be satisfied in Section \ref{section:equi}.

\begin{assumption}\label{assumption:equi} There exists a constant $\ep_0>0$ such that the following holds. 
  For any vector $r \in \R^2$ with $|r| = \frac{1}{2\pi \lambda}$,
  the set $\{\langle r, \mu \rangle, \mu \in \CE_{\lambda}\}$ can not be covered
  by less than $O(N^{\ep_0})$ 
  intervals of length $N^{-1}$ in $[-1,1]$.
\end{assumption}

Theorem \ref{theorem:general} is stated for almost all $m$ mainly because of the deterministic Lemma \ref{lemma:separation:d=2} of
Section \ref{section:smallinterval}, which in turn is needed for the
verification of one of our probabilistic conditions of the universality framework. We also need to
pass to almost all $m$ for a brief verification of Assumption
\ref{assumption:equi} above for $\CE_\lambda$.





Now we turn to $\T^d, d\ge 3$. While in this setting the cardinality $N$ of
$\CE_\lambda$ is relatively large compared to $\lambda$, the situation
is difficult by different reasons. Consider the following
example from \cite{RWY}.

\begin{example} \label{example3d} Let $F_0(x,y)$ be an eigenfunction on $\T^2$ with
  eigenvalue $4\pi^2 m$, and $S_0$ a curved segment length one
  contained in the nodal set, admitting an arc-length parameterization
  $\gamma_0:[0,1] \to S_0$ with curvature $\kappa_0(t) =
  |\gamma_0''(t)| >0$. For $n>0$, let $F_n(x,y,z)= F_0(x,y) \cos(2\pi
  n z)$, which is an eigenfunction on $\T^3$ with eigenvalue
  $4\pi^2(m+n^2)$. Let $\CC$ be the curve $\gamma(t) =
  (\gamma_0(t/\sqrt{2}), t/\sqrt{2})$. Standard computation shows that the curvature
  $\kappa(t) =\kappa_0(t/\sqrt{2})/2>0$ and the torsion $\tau(t)
  = \pm \kappa_0(t/\sqrt{2})/2$ is non-zero. Note that $\CC$ is contained in the nodal set of $F_n$ for all $n$. Thus we can have a non-trivial curve contained in the nodal set for arbitrary large $\lambda$.
\end{example} 

This example shows that the study of universality for discrete distributions in
$\T^d, d\ge 3$ can be highly complex (at least if we only assume $\gamma$ to have
non-vanishing curvature and torsion) as there is no deterministic upper bound for $\CZ(F)$. If we are not careful
with the choice of discrete distributions, our random function $F$ from
\eqref{eqn:F} might be one of the $F_n$ in Example \ref{example3d} with non-zero probability, and hence $\E
\CZ(F)$ is infinite. To avoid such type of singularity, in
what follows we will assume that the random variables $\ep_\mu$ satisfy Condition \ref{cond-variable}\eqref{cond-cont}. Note that this also holds for $d=2$.

\begin{theorem}[continuous distributions in
  $\T^d$, $d\ge 2$]\label{theorem:general:d>1} Assume that $\eps_{1,\mu},
  \eps_{2,\mu}, \mu \in \CE_\lambda$ are independent random
  variables  satisfying Condition \ref{cond-variable}\eqref{cond-cont}. Assume
  furthermore that the curve $\gamma$ extends analytically to the
  strip $[0, 1]\times [-\lambda^{-1}, \lambda^{-1}]$. Then for any
  fixed $k$ we have 
$$\E_{\ep_\mu} \CZ^k = \E_{\Bg} \CZ^k + O(\lambda^k/N^{c}).$$

In particularly for $\T^3$, with $\gamma$ and $\lambda$ as in Theorem \ref{theorem:gaussian:d=3}
$$\E_{\ep_\mu} \CZ =  \frac{2}{\sqrt{3}}\sqrt{m} + O(\lambda/N^{c}) \mbox{ and } \Var_{\ep_\mu}(\CZ) \ll \frac{\lambda^2}{N^c}.$$
\end{theorem}

The rest of the note is organized as follows. We first introduce  in Section
\ref{section:universality} a general scheme
from \cite{TV}, \cite{DOV} and \cite{OV} to prove our universality result, a sketch of proof for these results will
be discussed in Section \ref{pgcomplex}. In the next phase, we prove
Theorem \ref{theorem:general:d>1} for smooth distributions in Section \ref{section:smooth}. The
proof of Theorem \ref{theorem:general} will be carried out throughout Section
\ref{section:smallball}, \ref{section:smallinterval}, and \ref{section:equi} to check various regulatory conditions.

{\bf Notation.} We consider $\lambda$ as
an asymptotic parameter going to infinity and allow all other quantities to depend on $\lambda$ unless they are
explicitly declared to be fixed or constant.  We write $X =
O(Y)$, $X \ll Y$, or $Y \gg X$ if $|X| \leq CY$
for some fixed $C$; this $C$ can depend on
other fixed quantities such as the the parameter $K$ of
Condition 1 and the curvatures of $\gamma$. All the norms in this note, if not specified, will be the usual $\ell_2$-norm.

\section{Supporting lemmas: general universality results}\label{section:universality}

Our starting point uses the techniques
developed by T.Tao and V. Vu from \cite{TV}, and subsequently by Y. Do, O. Nguyen and V. Vu \cite{DOV} and by O. Nguyen and V. Vu \cite{OV}.

%
%

Let 
$$H(x) = \sum_{\mu\in \CE} \xi_\mu f_{\mu}(x),$$ 
where $x$ belongs to some set $\CB\subset \R$.


\begin{assumption}\label{condition:universality}  Consider the following conditions.

\begin{enumerate}
\item ({\it Analyticity})\label{analytic} $H$ has an analytic continuation on the set $\CB + B(0, 1)$ on the complex plane, which is also denoted by $H$.
\vskip .1in
\item ({\it Anti-concentration})\label{smallballcond} For any constants $A$ and $c$, there exists a constant $C$ such that for every $x\in \CB$, with probability at least $1-C N^{-A} $, there exists  $x'\in B(x, 1/100)$ such that $|H(x)| \ge \exp\left (-N^{c }\right )$.
\vskip .1in

\item ({\it Boundedness})\label{boundedness} For any constants $A$ and $c $, there exists a constant $C$ such that for every $x\in \CB$, 
$$\P\left (|H(z)| \le \exp\left (N^{c }\right ) \mbox{ for all $z\in B(x, 1)$}\right ) \ge 1- CN^{-A} .$$
\vskip .1in

\item ({\it Contribution of tail events})\label{tailex} For any $k\ge 1$, there exist constants $A, c>0$ such that for any $x\in \CB$ and any event $\mathcal A$ with probability at most $N^{-A}$, we have
$$\E \CZ_{B(x, 1)}^{k} \textbf{1}_{\mathcal A} = O_{k, A, c}(N^{-c}),$$
where $\CZ_{B(x, 1)}$ is the number of roots of $H$ in the complex
ball $B(x, 1)$.

 
\item ({\it Delocalization})\label{delocalization} There exists a constant $c>0$ such that for every $z\in \CB + B(0, 1)$ and every $\mu\in \CE$,
$$\frac{|f_\mu (z)|}{\sqrt{\sum_\mu f_\mu^2(z)}}\le N^{-c},$$
\vskip .1in

\item ({\it Derivative growth})\label{Dev} For any constant $c>0$, there exists a constant $C$ such that for any real number $x\in \CB + [-1, 1]$,
\begin{equation}\label{firstDev}
\sum_\mu |f_\mu'(x)|^2 \le C\left ( N^{c} \sum_\mu |f_\mu(x)|^2\right ),
\end{equation}
\vskip .1in
 as well as
\begin{equation}\label{secondDev}
\sup_{z\in B(x, 1)}|f_\mu''(z)|^2 \le C\left ( N^{c} \sum_\mu |f_\mu(x)|^2\right ).
\end{equation}
\end{enumerate}
\end{assumption}

Note that the last three conditions are deterministic, which are effective for trigonometric functions.  Now we state the main result from \cite{OV}.

\begin{theorem}[Local universality, real roots]\label{theorem:real:macro} Let $H(x)=\sum_{\mu} \xi_{\mu} f_\mu(x)$, with $H(x)$ be a random function with $f_{\mu}$ satisfying Assumption \ref{condition:universality}. Let $k$ be an integer constant. There exists a constant $c>0$ such that the following holds. For any real numbers $x_1,\dots, x_k$ in $\CB$, and for every smooth function $G$ supported on $\prod_{j=1}^{k} [x_j-c,x_j+c]$ with $|\nabla^a G(z)| \le 1$ for $0\le a \le 2k$ we have
\begin{equation}
\E_{\ep_\mu} \sum_{i_1,\dots,i_k} G (\zeta_{i_1},\dots, \zeta_{i_k}) -\E_{\Bg} \sum_{i_1,\dots,i_k} G (\zeta_{i_1},\dots, \zeta_{i_k}) = O(N^{-c}),\label{uni111}
\end{equation}
where the $\zeta_{i}$ are the roots of $H$, the sums run over all possible assignments of $i_1, \dots, i_k$ which are not necessarily distinct. 
\end{theorem}

\begin{remark}
By induction on $k$, the above theorem still holds if in \eqref{uni111}, the $i_1, \dots, i_k$ are required to be distinct.  
\end{remark}

We will provide a sketch of the proof of this theorem in Section \ref{pgcomplex}.


Now we consider  $F$ from \eqref{eqn:F}. Set the scaled function $H: [0,\lambda] \to \R$ to be
\begin{align}\label{eqn:H(x)}
H(x) :&=F\left (\frac{x}{\lambda}\right ) =\frac{1}{\sqrt N} \sum_{\mu\in \CE_\lambda} \eps_{\mu,1} \cos\left (2\pi  \left \langle \mu, \gamma\left (\frac{x}{\lambda}\right )\right \rangle\right ) + \eps_{\mu,2} \sin \left (2\pi  \left \langle \mu, \gamma\left (\frac{x}{\lambda}\right )\right \rangle\right ) \nonumber \\
&:= \frac{1}{\sqrt N}\sum_{\mu} \eps_{\mu,1} g_\mu(x)+ \eps_{\mu,2} h_\mu(x).\\
\nonumber \end{align}

Our main contributions are the following results.

\begin{theorem}\label{theorem:universality1} Under the assumptions of Theorem \ref{theorem:general:d>1}, let $\CB_1=[0, \lambda]$, then the function $H$ in \eqref{eqn:H(x)} satisfies the assumption (with $\CB = \CB_1$) and hence the conclusion of Theorem \ref{theorem:real:macro}. 
\end{theorem}

\begin{theorem}\label{theorem:universality2}
Under the assumptions of Theorem \ref{theorem:general}, let $\CB_2=[0,\lambda]\setminus \cup_{\varphi\in \CD} (\lambda S_\varphi)$ where $\CD$ is the set of directions 
$$\CD = \left \{\frac{\mu_1 - \mu_2}{\|\mu_1 -\mu_2\|}, \mu_1 \neq \mu_2, \mu_1, \mu_2\in \CE_\lambda\right \}.$$
and $$S_\varphi:=\{t\in [0,1], \angle(\gamma'(t), \varphi) <N^{-3}\}.$$
Then the function $H$ in \eqref{eqn:H(x)} satisfies the assumption (with $\CB = \CB_2$) and hence the conclusion of Theorem \ref{theorem:real:macro}.
\end{theorem}

We refer the reader to Section \ref{section:smallinterval} for the motivation of introducing $\CD$ and $S_\varphi$ as above. As a consequence, we have the following.

\begin{theorem}\label{theorem:subinterval:1&2}
Let $H$ be the function in \eqref{eqn:H(x)}. Under the assumptions of Theorem \ref{theorem:general:d>1} (respectively Theorem \ref{theorem:general}), for any $k\ge 1$, there exists a constant $c>0$ such that for any intervals $I_1, \dots, I_k\subset [0, \lambda]$ each belongs to $\CB_1$ (respectively $\CB_2$) and has length $O(1)$, we have 
$$\E_{\ep_\mu} \prod_{j=1}^k\CZ_{j} = \E_{\Bg}  \prod_{j=1}^k\CZ_{j} + O_k( N^{-c})$$
where $\CZ_j$ is the number of roots of $H$ in $I_j$.
\end{theorem}
To prove Theorems \ref{theorem:universality1} and \ref{theorem:universality2}, it suffices to verify all of the conditions of Assumption \ref{condition:universality} for $H(x)$. We will do so in Section \ref{section:smooth} for Theorem \ref{theorem:universality1} and Sections \ref{section:smallball} and \ref{section:smallinterval} for Theorem \ref{theorem:universality2}. 

Finally, the deduction of Theorem \ref{theorem:subinterval:1&2} from Theorems \ref{theorem:universality1} and \ref{theorem:universality2} is given in Section \ref{smoothappx}. Theorems \ref{theorem:general} and \ref{theorem:general:d>1} will be concluded from Theorem \ref{theorem:subinterval:1&2} in Section \ref{section:proofcut}.

\section{Proof of Theorem \ref{theorem:universality1}: the smooth case}\label{section:smooth}

Because of  \eqref{cond-cont} of Condition \ref{cond-variable}, we have the following anti-concentration bound.

\begin{fact}\label{fact:anti-concentration} For any $t\in I$, and any $\delta>0$
$$\P(|F(t)| \le \delta) =O(\delta).$$
\end{fact}

Our claim is that with very high probability all of the conditions from Assumption \ref{condition:universality} hold for the function $H$ given in \eqref{eqn:H(x)}. Note that Condition \eqref{analytic} of \ref{condition:universality} follows from our assumption on the analyticity of the curve $\gamma$.

\subsection{Verification of Condition \eqref{smallballcond}}\label{sub:smallballcond:smooth} For Condition \eqref{smallballcond}, if suffices to establish the bound for any $\mu_0\in \CE_\lambda$ and $x_0 \in J_i$. Again, as either $|\cos(2\pi \langle \mu_0, \gamma(x_0/\lambda) \rangle )|$ or $|\sin(2\pi \langle \mu_0, \gamma(x_0/\lambda) \rangle )|$ has order $\Theta(1)$, by the continuity of $\ep_\mu$, we have for any $\delta>0$,
\begin{eqnarray}
\P\left (|H(x_0)| \ge \delta\right ) &\ge& \inf_a \P\Big(\big |\eps_{1,\mu_0} \cos(2\pi \left \langle \mu_0, \gamma(x_0/\lambda) \right \rangle ) + \eps_{2,\mu_0} \sin(2\pi \langle \mu_0, \gamma(x_0/\lambda) \rangle ) + a\big |\ge N\delta \Big)\nonumber\\
&\ge& 1- O\left (N\delta\right )\label{anticoncentration1}.
\end{eqnarray}
Let $\delta = e^{-N^{c}}$, we obtain the desired estimate.

\subsection{Verification of Condition \eqref{boundedness}}\label{sub:doundedness:smooth}
For every $z\in [0, \lambda]\times [-1, 1]$, let $x = \Re(z)$. Since $\left \langle \mu, \gamma\left (\frac{x}{\lambda}\right )\right \rangle$ is real, we have
\begin{eqnarray}
\left |\Im   \left \langle \mu, \gamma\left (\frac{z}{\lambda}\right )\right \rangle\right |\le \left |\left \langle \mu, \gamma\left (\frac{z}{\lambda}\right )-\gamma\left (\frac{x}{\lambda}\right )\right \rangle\right | = O(1),
\end{eqnarray}
and so
\begin{eqnarray}
\left |\exp {\left (i2\pi  \left \langle \mu, \gamma\left (\frac{z}{\lambda}\right )\right \rangle\right ) }\right | = \exp {\left (-2\pi\Im   \left \langle \mu, \gamma\left (\frac{z}{\lambda}\right )\right \rangle\right ) } = O(1).\label{bound2}
\end{eqnarray}
Thus,
\begin{eqnarray}
|H(z)|=O(1)\sum_{\mu } |\ep_{\mu}|.\nonumber
\end{eqnarray}
By Markov's inequality, for any $M>0$,
\begin{equation}
\P\left (|H(z)| \ge M \mbox{ for some $z\in [0, T]\times [-1, 1]$}\right ) \le \P\left (\sum_{\mu } |\ep_{\mu}| =\Omega(M)\right )\le O\left (\frac{N}{M}\right ).\label{boundd1}
\end{equation}
Setting $M = e^{N^{c}}$, Condition \eqref{boundedness} then
follows. We remark that this condition holds even when $\ep_\mu$ has
discrete distribution.

\subsection{Verification of Condition \eqref{tailex}}\label{sub:tailexsmooth}
Let $K = \max_{z\in B(x, 1)}|H(z)|$. By Jensen's inequality,
\begin{eqnarray}
\CZ_{B(x, 1/2)}=O(1) \log \frac{K}{|H(x)|}\nonumber.
\end{eqnarray}
Thus,
$$\E \CZ_{B(x, 1/2)}^{k}\textbf{1}_{\mathcal A}\ll \E |\log K|^{k}\textbf{1}_{\mathcal A}+\E |\log |H(x)||^{k}\textbf{1}_{\mathcal A} .$$
By H\"older's inequality,
$$\E |\log K|^{k}\textbf{1}_{\mathcal A}\le \left (\E |\log K|^{2k}\right )^{1/2}\P(\mathcal A)^{1/2}.$$
By the bound \eqref{boundd1}, we obtain $\E |\log K|^{k} = O_{k}(N)$ which yields $$\E |\log K|^{k}\textbf{1}_{\mathcal A}=O_{k}\left (N^{-(A-1)/2}\right ).$$
We argue similarly for $ \E |\log |H(x)||^{k}\textbf{1}_{\mathcal A}$ using \eqref{anticoncentration1} (which is valid for all $\delta>0$). Letting $A = 2$, for example, we obtain the desired statement.

\subsection{Verification of Conditions \eqref{delocalization} and \eqref{Dev} for $g_\mu, h_\mu$} 
For Condition \eqref{delocalization}, note that for any $x\in (0,1)$ we have $\sum_{\mu} |g_\mu(x)|^2 + |h_\mu(x)|^2=N$, and so
$$\frac{|g_\mu(x)|+|h_\mu(x)|}{\sqrt{\sum_\mu g_\mu(x)^2 +h_\mu(x)^2 }} =O\left (\frac{1}{\sqrt{N}}\right ).$$
For \eqref{firstDev} of Condition \eqref{Dev}, 
%
we have
$$g_\mu'(x) = \frac{2\pi}{\lambda} \left \langle \mu, \gamma '(\frac{x}{\lambda})\right \rangle \cos\left (2\pi \left <\mu, \gamma \left (\frac{x}{\lambda}\right )\right >\right ).$$
Thus
$$\sum_\mu |g_\mu'(x)|^2 + \sum_\mu |h_\mu'(x)|^2 \ll \sum_{\mu} \frac{1}{\lambda^{2}} \left \langle \mu, \gamma '(\frac{x}{\lambda})\right \rangle^{2}\ll N$$
where the implicit constant depends on $\max_{x\in [0,\lambda]}
|\gamma'(\frac{x}{\lambda})|$. This proves 
\eqref{firstDev}. Finally, \eqref{secondDev} of Condition \eqref{Dev} is proven similarly using the same argument together with \eqref{bound2}.


In the remaining sections we will prove Theorem
\ref{theorem:universality2}. As we already seen, for this it suffices to verify
Condition \eqref{smallballcond} and Condition \eqref{tailex} of Assumption
\ref{condition:universality} only.

\section{Proof of Theorem \ref{theorem:universality2}: verification of Condition \eqref{smallballcond}}\label{section:smallball}

As the continuous case has been treated in Section \ref{section:smooth}, here we will assume 

\begin{itemize}
\item there exist positive constants $c_1,c_2,c_3$ and $K$ such that
$$\P(c_1\le |\ep_\mu - \ep_\mu'|\le c_2)\ge c_3$$
\item with probability one
$$|\ep_\mu|>1/K$$ 
\end{itemize}

Recall that $N=|\CE_\lambda|$. Without scaling, we will show the following which implies Condition \eqref{smallballcond} of Assumption \ref{condition:universality}.

\begin{theorem}\label{theorem:smallball} Let $A>0$ be a fixed constant, then there exists a constant $C=C(A)$ such that the following holds for $F(t)$ from \eqref{eqn:F}: for any interval $I\subset [0,1]$ of length $c_0/\lambda$, for any $t_1, t_2\in I$ with $\|\gamma(t_1)-\gamma(t_2)\| = \frac{N^{-\alpha}}{\lambda}$, we have
$$\P(\left |F(t_1)\right | \le N^{-C}) \le N^{-A}\quad\mbox{or}\quad \P(\left |F(t_2)\right | \le N^{-C}) \le N^{-A}.$$
\end{theorem}
Note that by the remark after Condition \ref{cond-curve}, for any interval $I$ of length $c_0/\lambda$, there exist $t_1, t_2\in I$ with $\|\gamma(t_1)-\gamma(t_2)\| = \frac{N^{-\alpha}}{\lambda}$.

It is clear that Condition \eqref{smallballcond} of Assumption \eqref{condition:universality} follows immediately where the sub-exponential lower bound can be replaced by polynomial bounds. To prove Theorem \ref{theorem:smallball} we will rely on two
results on additive structures. We say a set $S\subset\mathbb C$ is $\delta$-{\it separated} if for any $s_1, s_2\in S$, $|s_1-s_2|\ge\delta$, and $S$ is $\eps$-{\it
close} to a set $P$ if for all $s\in S$, there exists $p\in P$ such
that $|s-p|\le \eps$.

 Define a \emph{generalized arithmetic progression} (or GAP) to be a finite subset $Q$ of $\C$ of the form
$$ Q = \{ g_0+a_1 g_1 + \dots + a_r g_r: a_i \in \Z, |a_i| \leq N_i \hbox{ for all } i=1,\dots,r\}$$
where $r \geq 0$ is a natural number (the \emph{rank} of the GAP),
$N_1,\dots,N_r > 0$ are positive integers (the \emph{dimension lengths}, or \emph{dimension} for short, of
the GAP), and $g_0, g_1,\dots,g_r \in \C$ are complex numbers (the
\emph{generators} of the GAP).  We refer to the quantity
$\prod_{i=1}^r (2N_i+1)$ as the \emph{volume} $\operatorname{vol}(Q)$
of $Q$; this is an upper bound for the cardinality $|Q|$ of $Q$. When
$g_0=0$, we say that $Q$ is {\it symmetric}. When $\sum_i a_ig_i$ are
all distinct, we say that $Q$ is {\it proper}.

Let $\xi$ be a real random variable, and let $V = \{v_1, . . . , v_n\}$
be a multi-set in $\R^{d}$. For any $r >0$, we define the small ball probability as
$$\rho_{r,\xi}(V) := \sup_{x\in \R^{d}}\P\left (v_1 \xi_1 + \dots + v_n \xi_n \in B(x, r)\right )$$
where $\xi_1, . . . , \xi_n$ are iid copies of $\xi$, and $B(x, r)$ denotes the closed disk of radius r centered at
$x$ in $\R^{d}$.

\begin{theorem}\cite[Theorem 2.9]{NgV}\label{theorem:ILO} Let $ A>0$
  and $1/2>\eps_0>0$ be constants. Let
$ \beta >0$ be a parameter that may depend on $n$. Suppose that $V=\{v_1,\dots,v_n\}$ is a (multi-) subset of $\R^d$ such that $\sum_{i=1}^n\|v_i\|^2=1$ and that $V$ has large small ball probability
$$\rho:= \rho_{\beta,\xi}(V)\ge n^{-A}, $$ where $\xi$ is a real  random variable satisfying Condition \ref{cond-variable}. Then the following holds: for any number $ n^{\eps_0} \le n'  \le n$, there exists a proper symmetric GAP $Q=\{\sum_{i=1}^r x_ig_i : |x_i|\le L_i \}$ such that

\begin{itemize}
\item At least $n-n'$ elements of $V$  are $O(\beta)$-close to $Q$.
\vskip .05in
\item  $Q$ has constant rank $d \le r=O (1)$, and cardinality
$$|Q| =O(\rho^{-1} n'^{(-r+d)/2}).$$



\end{itemize}
\end{theorem}


For Theorem \ref{theorem:smallball}, first fix $t\in I$, and
let $x=\gamma(t)$. Set $\beta= N^{-C}$, with $C$ sufficiently large to be chosen, and assume that
\begin{equation}\label{eqn:anti:t}
\P \left (\left |\sum_{\mu\in \CE_\lambda} \eps_{1,\mu} \cos (2\pi \langle \mu,
x\rangle) +  \eps_{2,\mu} \sin (2\pi \langle \mu, x\rangle)\right |\le
\beta \right ) \ge N^{-A}.
 \end{equation}

We will choose $\ep_0$ to be the constant in Assumption \ref{assumption:equi}. Then by Theorem \ref{theorem:ILO} (applied to the sequences $\{\cos
(2\pi \langle \mu, x\rangle), \mu \in \CE_\lambda\}$ and $\{\sin (2\pi
\langle \mu, x\rangle), \mu \in \CE_\lambda\}$ separately with $N'=N^{\eps_0}$), there exist
proper GAPs $P_1,P_2 \subset \R$ and $|\CE_\lambda|-2N'$ indices $\mu \in \CE_\lambda$
such that with $z_\mu(t)=\cos (2\pi \langle \mu, x\rangle) + i  \sin
(2\pi \langle \mu, x\rangle)=\exp(2\pi i \langle \mu, \gamma(t)\rangle)$,
$$ \dist (z_\mu(t), P_1 + iP_2) \le 2 \beta $$
and such that the cardinalities of $P_1$ and $P_2$ are $O\left (N^{O_A(1)}\right )$ and the ranks are $O(1)$. The properness implies that the dimensions of the GAPs $P_1$ and $P_2$ are bounded by $O\left (N^{O_A(1)}\right )$.


For short, we denote the complex GAP $P_1+i P_2$ by $P(t)$.

Now assume for contradiction that \eqref{eqn:anti:t} holds for both $t = t_1$ and $t=t_2$. 
By applying the above process to $t_1$ and $t_2$, we obtain two GAPs
$P(t_1)$ and $P(t_2)$ which are $2\beta$-close to
the points $z_\mu(t_1)$ and $z_\mu(t_2)$ respectively for at least
$N-4N^{\eps_0}$ indices $\mu$.

Since the $z_\mu(t_1)$ and $z_\mu(t_2)$ have magnitude 1, the
product set
$P(t_1)  \bar{P}(t_2)=\{p_1 \bar{p_2}, p_1\in P_1(t), p_2 \in P_2(t)\}$ will $O(\beta
)$-approximate the points $z_\mu=z_\mu(t_1) \bar{z}_\mu(t_2) =
\exp(2\pi \langle \mu, \gamma(t_1)-\gamma(t_2)\rangle)$  for at least
$N-4N^{\eps_0}$ indices $\mu$. Let $\mathcal S$ be the collection of these points $z_\mu$.

By definition, $P=P(t_1) \bar{P}(t_2)$ is another GAP whose rank is $O(1)$ and dimensions are of order $O\left (N^{O_A(1)}\right )$.

Now we look at the set $\mathcal S$. 
 On one hand, $\mathcal S$ is ``stable" under multiplication in the sense that $|z_{\mu_1} z_{\mu_2}|=1$ for all $\mu_1,\mu_2$. On the other hand, as $z_{\mu}$ can be well approximated by elements of a GAP of small size, the collection of sums $z_{\mu_1} + z_{\mu_2}$ can also be approximated by another GAP  of small size. Roughly speaking, in line of the "sum-product" phenomenon in additive combinatorics \cite{ESz}, this is only possible if the GAP sizes are extremely small.  Rigorously, we will need the following continuous analog of a result by the first author \cite{C}.

\begin{theorem}\label{theorem:circle} Let $P =\{g_0+\sum_{i=1}^r n_i g_i: |n_i|< M\}$ be a generalized arithmetic progression of rank $r$ on the complex plane. Then there exists an (explicit) constant $C_r$ with the following property. Let $0<\delta <1$ and $\eps < M^{-C_r} \delta^{C_r}$ and let $S \subset P$ be a subset consisting of elements  which are $\delta$-separated and $\eps$-close to the unit circle, then
$$S \le \exp(C_r \log M/\log \log M).$$
\end{theorem}
To complete the proof of Theorem \ref{theorem:smallball}, we apply Theorem \ref{theorem:circle} 
with $\eps = O(\beta)$, $r=O_A(1)$, and $M=O(N^{O_A(1)})$
to conclude that the set $\mathcal S$ can be covered by $\exp(C_r \log N/\log \log N)$
disks of radius $\delta$ with $\delta= M \eps^{1/C_r}$. Taking into account at most $4N^{\ep_0}$ elements $z_\mu$ not included in $\mathcal S$, the set $ \{\langle
\mu,  \gamma(t_1)-\gamma(t_2) \rangle\}_{\mu\in \mathcal E}$ can be covered by $4N^{\ep_0}+\exp(C_r \log
N/\log \log N)\le 5N^{\ep_0}$ intervals of length $O(\delta)$. However, note that
$$\delta = M \eps^{1/C_r} =O\left (N^{-C/C_r + O_A(1)}\right ).$$
By choosing $C$ sufficiently large, this would contradict with the
equi-distribution assumption \ref{assumption:equi} on $\CE$.





For the rest of this section we will justify  Theorem \ref{theorem:circle}. In this proof, $C_r$ is  a constant depending on $r$ and may vary even within the same context. We denote the set of the coefficient vectors of $S$ by
$$\CF = \left \{\bar{n} = (n_1,\dots, n_r) \in \mathbb Z^r : |n_i|<M, g_0+\sum_{i=1}^r n_ig_i \in S\right\}.$$
Fix $\bar m\in \CF$. Since $g_0+\sum_{i=1}^r m_ig_i$ is $\ep$-close to the unit circle, we have $|g_0+\sum_{i=1}^r m_ig_i|\le 1+\ep$ and 

\begin{equation}\label{circle:2}\bigg|\sum_{i=1}^r(n_i-m_i)g_i\bigg|\le 2(1+\eps) \;\; \text{
  for all  } \bar n\in\mathcal F.
\end{equation}

Let $\langle\mathcal F -\bar m \rangle$ be the vector space generated by $\bar{n} -\bar{m}, \bar{n}\in \CF$. We assume $\dim\langle \mathcal F- \bar{m}\rangle=r$, since otherwise we may reduce the rank of $P$ without significantly changing the size of $P$
 (see \cite[Chapter 3]{TVbook}).

 Therefore, we can take $r$ independent vectors $\bar n^{(1)},\cdots, \bar n^{(r)}\in\mathcal F$ and use Cramer's rule to solve $g_1, \cdots, g_r$ in the following system of $r$ equations.
$$\begin{aligned}
(n_1^{(1)}-m_1)g_1+&\cdots+ (n_r^{(1)}-m_r)g_r =c^{(1)}\\
&\cdots\\
&\cdots\\
&\cdots\\
(n_1^{(r)}-m_1)g_1+&\cdots+ (n_r^{(r)}-m_r)g_r =c^{(r)}\end{aligned}$$
where $|c^{(1)}|,\cdots, |c^{(r)}|\le 2 (1+\eps)<3$.




We obtain a bound

\begin{equation}\label{circle:3}|g_1|,\dots, |g_r| \le 3.2^{r} r! M^{r-1},
\end{equation}
and hence

\begin{equation}\label{circle:4}
|g_0| < \sum_i |n_i g_i| +1 +\eps < (3r)2^{r} r! M^r.
\end{equation}

Next, assume that $|\CF|\ge 2$. Then the separation assumption means
that for any $\bar{m}, \bar{n}\in \CF$ with $\bar{m} \neq \bar{n}$ we
have $|\sum_{i=1}^r (m_i-n_i)g_i| >\delta$. Thus,

\begin{equation}\label{circle:5}
\max\{|g_1|,\dots, |g_r|\} > \frac{\delta}{2r M}.
\end{equation}

Without loss of generality, assume that the maximum above is attained by $|g_1|$.

\smallskip

\begin{lemma}\label{lemma:wz}There exist $z_0,z_1,\dots, z_r, w_0,w_1,\dots, w_r \in \mathbb C$ with $z_1 \neq 0$ such that for any $\bar{n} \in \CF$
$$\Big(z_0+\sum_{i=1}^r n_i z_i\Big)\Big(w_0+\sum_{i=1}^r n_i w_i\Big)=1.$$
\end{lemma}

\medskip

We next conclude Theorem \ref{theorem:circle} using this lemma. Let $\CA= \{z_0+ \sum_{i=1}^r n_i z_i: \bar n\in\mathcal F\}$. 

Applying Proposition 3 in \cite{C} to the mixed progression $$\{n_0z_0+n_0w_0+\sum_{i=1}^r n_i z_i + \sum_{i=1}^r n_i' w_i: |n_0|, |n_0'| <2 \text{ and } |n_i|, |n'_i|<M\},$$ we have
$$|\CA| \le \exp(D_r \log M /\log \log M),$$
for some positive constant $D_r$.

We next partition $\CF$ as
$$\CF = \bigcup_{a\in \CA} \CF_a, \mbox{ where } \CF_a =\Big\{ \bar{n} \in \CF: z_0+\sum_{i=1}^r n_iz_i=a \Big\}.$$
Let $S$ be as in Theorem \ref{theorem:circle}, we write

\begin{equation}\label{circle:6}
S= \Big\{g_0 + \sum_{i=1}^r n_i g_i : \bar{n} \in \CF\Big\} =
\bigcup_{a\in \CA} S_a,
\end{equation}

where
$$S_a:= \{g_0 + \sum_{i=1}^r n_i g_i : \bar{n} \in \CF_a\}.$$
Notice that $S_a \subset P_a:=\{g_0 + \sum_{i=1}^r n_i g_i \in P: z_0 + \sum_{i=1}^r n_iz_i =a\}.$ The gain here is that $P_a$ is contained in a progression of rank at most $r-1$, that is, 
$$g_0 + \sum_{i=1}^{r} n_i g_i = \left (g_0 + \frac{a-z_0}{z_1}g_1\right )+ \sum_{i=2}^{r} n_i\left (g_i-\frac{z_i}{z_1}g_1\right )$$
 so by induction
$$|S_a| \le \exp(C_{r-1} \log M/\log \log M).$$
It thus follows from \eqref{circle:6} that
$$|S| \le \exp (C_r \log M/\log \log M),$$
for some appropriately chosen constant sequence $C_r$, completing the
proof of Theorem \ref{theorem:circle}.

We now prove Lemma \ref{lemma:wz}. We will use the following effective form of Nullstellensatz  \cite{KPS}.

\begin{theorem}\label{theorem:KPS}Let $q,f_1,\dots,f_s \in \mathbb Z[x_1,\dots,x_n]$ with $\deg q, \deg f_i \le d$ for all $i$ such that $q$ vanishes on the common zeros of $f_1,\cdots, f_s$ and $ \height(f_i) \le H$. Then there exist $q_1,\dots,q_s \in \mathbb Z[x_1,\dots,x_n]$ and positive integers $b,l$ such that
\begin{equation}
b \;q^l = \sum_{i=1}^s q_i f_i\label{null2}
\end{equation}
where
$$l\le D =\max_{1\le i\le s}\{\deg q_i \} \le 4n d^n$$
as well as
$$\max_{1\le i\le s}\{\log |b|,  \height(q_i) \} \le 4 n(n+1)d^n
\big[H+\log s + (n+7)d \log (n+1)\big].$$
\end{theorem}
Here the height $\height(f)$ of a polynomial $f\in \Z[x_1, \dots, x_n]$ is the logarithm of the maximum modulus of its coefficients.

\noindent {\bf Remark.} Theorem 1 in \cite{KPS} is stated for the case that $q=1$ and  that $f_1,\dots, f_s$ do not have common zeros. However, the standard proof of Nullstellensatz gives the above statement (see \cite[Proposition 9]{BBK} for instance.)

Now define the polynomial $P$ over $\bar{n} \in \CF$ as
$$P_{\bar{n}} (z_0,z_1,\dots, z_r, w_0,w_1,\dots, w_r) = \big(z_0+\sum_{i=1}^r n_i z_i\big) \big(w_0+\sum_{i=1}^r n_i w_i\big) -1.$$
Assume that the claim of Lemma \ref{lemma:wz} does not hold, thus the polynomials $P_{\bar{n}}, \bar{n}\in \CF$ have no common zeros with $z_1 \neq 0$. 

By Theorem \ref{theorem:KPS}, with $n=2r+2, s = |\CF| \le (2M)^{r}, d=2, H \le 2 \log M$ we have
\begin{equation}\label{circle:7}
b z_1^l = \sum_{\bar{n}\in \CF} P_{\bar{n}} Q _{\bar{n}},
\end{equation}
where $b \in \mathbb Z\backslash \{0\}$, $Q_{\bar{n}} \in \mathbb Z[z_0,\dots, z_r, w_0,\dots, w_r]$ such that

\begin{itemize}
\item $\deg(Q_{\bar{n}}) , l\le D \le C_r'$
\vskip .1in
\item the coefficients of $Q_{\bar{n}}$ are bounded by $M^{C_r'}$.
\end{itemize}

Now replacing $z_0,\dots, z_r$ and $w_0,\dots, w_r$ by $g_0,\dots, g_r$ and $\bar{g}_0,\dots, \bar{g}_r$ in \eqref{circle:7}, we have
 $$|g_1|^l \le  \sum_{\bar{n}\in \CF} |P_{\bar{n}}(g_0,\dots, g_r, \bar{g}_0,\dots, \bar{g}_d)|\;
 |Q _{\bar{n}}(g_0,\dots, g_r, \bar{g}_0,\dots, \bar{g}_d)|.$$
By \eqref{circle:3}, \eqref{circle:4}, \eqref{circle:5}  we then have
$$\bigg(\frac{\delta}{2r M}\bigg)^l \le D M^{C_r'} (3. 2^{r} r! r M^r)^D
\sum_{\bar{n}\in \CF} |P_{\bar{n}} (g_0,\dots, g_r, \bar{g_0},\dots, \bar{g}_r)|.$$
On the other hand, by definition, $|P_{\bar{n}} (g_0,\dots, g_r, \bar{g_0},\dots, \bar{g}_r)| \le \eps$ for any $\bar{n} \in \CF$. It thus follows that
 $$\bigg(\frac{\delta}{2r M}\bigg)^l\le\bigg(\frac{\delta}{2r M}\bigg)^D \le M^{C_r''} \eps.$$
However, this is impossible with the choice of $\eps$ from Theorem \ref{theorem:circle}.

\section{Proof of Theorem \ref{theorem:universality2}: verification of Condition \eqref{tailex}}\label{section:smallinterval}

Let $\kappa=N^{-3}$. 
We will verify Condition \eqref{tailex} of Assumption \ref{condition:universality} through the following deterministic lemma, which is of independent interest.

\begin{theorem}\label{theorem:deterministic:d=2} Suppose that $\gamma(t), t\in [0,1]$ is smooth and has non-vanishing curvature. Then there exist a constant $c$ and a collection of at most $N^2$ intervals $S_\alpha$  each of length $O(\kappa)$ such that the following holds for almost all $\lambda$ and for any eigenfunction $\Phi(x) = \sum_{\mu\in \CE_\lambda} a_\mu e^{2\pi i \langle \mu, x \rangle }$ with $\sum_\mu |a_\mu|^2 =1$.
\begin{enumerate}
\item The number of nodal intersections on $\cup S_{\alpha}$ is negligible
$$|N_{\Phi} \cap \cup \gamma(S_\alpha)| \ll \lambda N^{-1},$$ 
\vskip .1in
\item Condition \eqref{tailex} on $[0,1]\setminus \cup S_{\alpha}$: for any $a\in [0,1] \backslash \cup_\alpha S_\alpha$, we have
$$|\{z \in B(a, N^{7}/\lambda): \Phi(\gamma(z))=0\}|\ll N^7.$$
\end{enumerate}
\end{theorem}

To prove Theorem \ref{theorem:deterministic:d=2} we first need a separation result (see also \cite[Lemma 5]{BR}).

\begin{lemma}\label{lemma:separation:d=2} For almost all $\lambda$, we have
\begin{equation}
\min_{\mu_1 \neq \mu_2 \in \CE_\lambda} \|\mu_1-\mu_2\| \gg \frac{\lambda}{\log^{3/2+\eps} \lambda}.\label{dist1}
\end{equation}
\end{lemma}

\begin{proof}(of Lemma \ref{lemma:separation:d=2}) Let $R$ be a parameter and $M = R (\log R)^{-3/2-\eps}$. Then 
\begin{align*}
&\Big|\{(x,y)\in \Z^2 \times \Z^2: \|x\| = \|y\| \le R, 0<\|x-y\| <M \}\Big| \\
&=
\sum_{v \in \Z^2\backslash \{0\}, \|v\|< M} |\{x\in \Z^2: \|x\| = \|x+v\| \le R\}|\\
&= \sum_{v \in  \Z^2\backslash \{0\}, \|v\|< M} \Big|\{\|x\|\le R: 2 \langle x, v  \rangle + \|v\|^2 =0\}\Big|\\
&\le \sum_{v \in  \Z^2\backslash \{0\}, \|v\| <M} \Big|\{\|y\| \le 3R: y_1 v_1+ y_2v_2=0\}\Big|,
\end{align*}

where $x=(x_1,x_2), v=(v_1,v_2)$ and $y=(y_1,y_2)=2x+v$.

Now if $v_2=0$ then $y_1=0$. The contribution to the above sum is $O(MR)$. Similarly for $v_1=0$. For the other case that $v_1, v_2 \neq 0$, let $d=\gcd(v_1,v_2)$. Then $(v_1,v_2)= d(v_1',v_2') $ with $ \gcd(v_1',v_2')=1$. The equation $y_1 v_1' + y_2 v_2'=0$ has $O\left (R/\|v'\|\right )$ solutions in $y$ with $\|y\| < 3R$.

So by the Abel's summation formula, we have
\begin{align*}
&\sum_{v \in  \Z^2\backslash \{0\}, \|v\| <M} |\{\|y\| \le 3R: y_1 v_1+ y_2v_2=0\}| \ll  MR + \sum_{d<R} \sum_{v' \in \Z^2\backslash \{0\},\|v'\|<M/d} R/\|v'\|\\
&= R \sum_{d<R} \sum_{n=1}^{M^2/d^2}\frac{r_2(n)}{\sqrt{n}} = R \sum_{d<R}\left [ \frac{\sum_{n=1}^{M^2/d^2}r_2(n)}{M/d} + \sum_{N=1}^{M^2/d^2-1} (\sum_{n=1}^{N}r_2(n))\left (\frac{1}{\sqrt{N}} - \frac{1}{\sqrt{N+1}}\right ) \right ].
\end{align*}
By Gauss' formula
$$\sum _{n=0}^{x} r_2(n) = (\pi+o(1))x,$$
we have
\begin{equation}
\sum_{v \in  \Z^2\backslash \{0\}, \|v\| <M} |\{\|y\| \le 3R: y_1 v_1+ y_2v_2=0\}|\ll R\sum_{d<R} \frac{M}{d} \ll MR\log R.\nonumber
\end{equation}

Hence 
$$|\{(x,y)\in \Z^2 \times \Z^2: \|x\| = \|y\| \le R, 0<\|x-y\| <M \}|  \ll MR \log R.$$
On the other hand, 
$$|\{(x,y)\in \Z^2 \times \Z^2: \|x\| = \|y\| \le R, 0<\|x-y\| <M \}| \ge \sum'_{ E < R^2} \textbf{1}_{\min\{\|x-y\|, \|x\|^2 = \|y\|^2 = E, x \neq y\}<M},$$
where $\sum' $ is the sum over $E$ of sum of two squares. Note that by a classical result of Landau \cite{Landau}
$$|\{E \in \Z, E <R^2, E=\mbox{sum of two squares}\}| \gg R^2 /\sqrt{\log R}.$$

Recall that $M = R (\log R)^{-3/2 -\eps}$. 
Thus for almost all $E \le R^2$ that are sum of two squares, 
$$\min_{\|x\|^2 =\|y\|^2 =E, x \neq y}\|x-y\|\ge M \gg R  (\log R)^{-3/2 -\eps} \gg \sqrt{E} (\log E)^{-3/2 -\eps}.$$
\end{proof}

Recall that by Condition \ref{cond-curve}\eqref{cond-ana}, the curve $\gamma$ has an analytic continuation to $[0,1]+
B(0,\ep) \subset \C$. Arguing as in Sections \ref{sub:doundedness:smooth} and \ref{sub:tailexsmooth}, we get the following.

\begin{lemma}\label{lemma:Jensen:d=2}  Let $I$ be any interval with length $\delta =|I| <\ep/2$. Then for any $\Phi$ as in Theorem \ref{theorem:deterministic:d=2}
$$|\{z \in I+B(0, \delta): \Phi(\gamma(z))=0\}| \le C \lambda \delta +\log N - \log \max_{t\in I} |\Phi(\gamma(t))|.$$
\end{lemma}

\begin{proof}(of Lemma \ref{lemma:Jensen:d=2}) For $z\in I+B(0, 2\delta ), \exists t \in \R$ such that  $|z-t| < 2\delta $, 
$$|\gamma(z) - \gamma(t)| \le c \delta.$$

Hence for $\mu \in \CE_\lambda,$
$$\left |e^{i  \langle \mu, \gamma(z) \rangle }\right |= \left |e^{i  \langle \mu, \gamma(z)-\gamma(t)\rangle }  \right |\le e^{c \lambda \delta}.$$
Therefore 
$$|\Phi(\gamma(z))| \le (\sum_{\mu \in \CE_\lambda} |a_{\mu}|) e^{c \lambda \delta} < \sqrt{N} e^{c \lambda \delta}.$$
Jensen's inequality then implies
\begin{align*}
|\{z \in I+B(0, \delta), \Phi(\gamma(z))=0\}| &\le \log (\sqrt{N} e^{c \lambda \delta}) - \log \max_{t\in I} |\Phi(\gamma(t))| \\
&\le c \lambda \delta + \log N -  \log \max_{t\in I} |\Phi(\gamma(t))|. 
\end{align*}
\end{proof}

Now we want to bound $\max_{t\in I} |\Phi(\gamma(t))|$.
\begin{lemma}\label{lemma:L2:d=2} We have 
$$\frac{1}{|I|} \int_I |\Phi(\gamma(t))|^2 dt  \ge 1/2,$$

provided that $\lambda$ satisfied \eqref{dist1} of Lemma \ref{lemma:separation:d=2} and 
$$|I| > \lambda^{-1/2} (\log \lambda)^{3/4+\eps} N.$$
\end{lemma}

\begin{proof}(of Lemma \ref{lemma:L2:d=2}) 
We write 
\begin{align*}
\int_I | \Phi(\gamma(t))|^2 dt = \int_I \left |\sum_\mu a_\mu e^{2\pi i \langle \mu, \gamma(t) \rangle }\right |^{2} dt&= |I| +\sum_{\mu \neq \mu'} a_\mu \bar{a}_{\mu'} \int_I  e^{2\pi i \langle \mu-\mu', \gamma(t) \rangle }\\
&\ge |I| - \sum_{\mu \neq \mu'} |a_\mu| |a_{\mu'}| |\int_I  e^{2\pi i \langle \mu-\mu', \gamma(t) \rangle }|.
\end{align*}

By van der Corput's lemma on oscillatory integral (see for instance \cite{BR2}),
$$\left |\int_I  e^{2\pi i \langle \mu-\mu', \gamma(t) \rangle }dt\right | \le \frac{1}{\|\mu -\mu'\|^{1/2}}.$$ 
Hence
$$\int_I | \Phi(\gamma(t))|^2 dt  \ge |I| - \frac{\log^{3/4+\eps} \lambda}{\lambda^{1/2}} N \gg |I|/2.$$
\end{proof}

Recall the set of directions,
$$\CD = \left \{\frac{\mu_1 - \mu_2}{\|\mu_1 -\mu_2\|}, \mu_1 \neq \mu_2, \mu_1, \mu_2\in \CE_\lambda\right \}.$$

We partition $[0,1]$ as follows: for every unit direction $\varphi$,
let $S_\varphi$ be the interval 
$$S_\varphi:=\{t\in [0,1], \angle(\gamma'(t), \varphi) <\kappa\}.$$
\begin{claim}\label{claim:interval:d=2} Assume that the arc-length parametrized curve $\gamma(t)$ has curvature bounded from below by some $c>0$ for all $t$. Then for each $\varphi$, $S_\varphi$ is an interval and has size $O(\kappa)$, where the implied constant depends on $c$.
\end{claim}

\begin{proof} Let $a(t)$ be the angle between $\gamma'(t)$ and $\varphi$. Then the curvature of $\gamma$ at $t$ is $|a'(t)|$ by definition. 
By continuity, the assumption that $\gamma$ has curvature bounded from below by $c$ implies that either $a'(t) \ge c$ for all $t$ or $a'(t)\le -c$  for all $t$.  From either case, it is easy to deduce the claim.


\end{proof}


Let $J = [0, 1]\setminus \cup_{\varphi\in \CD} S_\varphi$. We note that $J$ depends on $\CE_\lambda$ and $\gamma$ but not on $\Phi$. Now we prove Theorem \ref{theorem:deterministic:d=2}. We first show that $|N_\Phi \cap \cup \gamma(S_\varphi)| \le \lambda N^{-1}$.

Note that as $\kappa > \lambda^{-1/2} (\log \lambda)^{3/4+\eps} N$, the condition of Lemma \ref{lemma:L2:d=2} holds.  Thus 
$$\max_{t\in S_\varphi} |\Phi(\gamma(t))| \ge \frac{1}{|S_\varphi|} \int_{S_\varphi} |\Phi(\gamma(t))|^2 dt  \ge 1/2.$$
Lemma \ref{lemma:Jensen:d=2} implies that 
$$|N_\Phi \cap \gamma(S_\varphi)| \ll \kappa \lambda + \log N -c \ll \kappa \lambda.$$
Hence
$$|N_\Phi \cap \cup_\varphi \gamma(S_\varphi)| \ll N^2 \kappa \lambda \ll \lambda N^{-1}$$
proving the first part of Theorem \ref{theorem:deterministic:d=2}.

Now for the second part, let $a\in J$. Let $\delta = N^{7}/\lambda$, $M=N^{7}$.

Denote $\tilde{I} = [a-\delta,a+\delta]$. 
 Again, Lemma \ref{lemma:Jensen:d=2} implies that for $\delta =M/\lambda  \le \lambda^{-1+\eps}$
$$|\{z \in B(a, \delta): \Phi(\gamma(z))=0\}|\le |\{z \in \tilde I + B(0, \delta): \Phi(\gamma(z))=0\}|\le c M + \log N -\log \max_{t \in \tilde{I}} |\Phi(\gamma(t))|.$$
Since $a \in J, \angle(\gamma'(a), \varphi) \ge \kappa, \forall \varphi \in \CD $. Thus for any $\mu \neq \mu'$, 
$$|\langle \mu-\mu', \gamma'(a) \rangle | \ge \kappa \|\mu -\mu'\| \gg \delta \|\mu -\mu'\|.$$
On the other hand, with $\delta =M/\lambda  \le \lambda^{-1+\eps}$ and $t=a+\tau$, write 
$$\langle \mu-\mu' ,\gamma(t) \rangle = \langle \mu-\mu', \gamma(a) \rangle + \langle \mu-\mu', \gamma'(a) \tau \rangle + O(\|\mu-\mu'\| \delta^2).$$

Because $|\langle \mu-\mu', \gamma'(a)  \rangle| \ge \kappa \|\mu -\mu'\| \gg \|\mu-\mu'\| \delta$ and  $\|\mu -\mu'\| \delta^2 \ll \lambda \lambda^{-2+\eps}\ll \lambda^{-1+\eps}$,
$$\left |\int_{-\delta}^\delta e^{i  \langle \mu-\mu', \gamma'(a) \tau \rangle } d\tau \right | \le \frac{1}{ |\langle \mu-\mu', \gamma'(a)  \rangle|}.$$
We thus have
$$\frac{1}{|\tilde{I}|} \left |\int_{\tilde{I}}   e^{i  \langle (\mu-\mu'), \gamma'(a) \tau \rangle } d\tau \right | \le \frac{1}{ \delta |\langle (\mu-\mu'), \gamma'(a)  \rangle|} \le \frac{1}{\delta \kappa \|\mu-\mu'\|} \le \frac{\lambda}{M \kappa \|\mu -\mu'\|}.$$
Lemma \ref{lemma:separation:d=2} says that $\|\mu-\mu'\| \gg \frac{\lambda}{\log^{3/2+\eps} \lambda}$. Hence
$$\frac{1}{|\tilde{I}|} \left |\int_{\tilde{I}}   e^{i  \langle (\mu-\mu'), \gamma(t) \rangle } dt \right |\le \frac{1}{|\tilde{I}|} \left |\int_{\tilde{I}}   e^{i  \langle (\mu-\mu'), \gamma'(a) \tau \rangle } d\tau \right |+ O(\|\mu-\mu'\| \delta^2) \le \frac{N^{3}\log^{3/2+\eps} \lambda }{M}.$$
Now we have 
$$\frac{1}{|\tilde{I}|} \left |\int_{\tilde{I}}   |\Phi(\gamma(t))dt\right |^2  \ge 1 - \sum_{\mu \neq \mu'} |a_\mu||a_\mu'| \frac{1}{|\tilde{I}|} |\int_{\tilde{I}}   e^{i  \langle (\mu-\mu'), \gamma(t) \rangle } dt | \gg 1.$$
So 
$$\max_{t \in \tilde{I}} |\Phi(\gamma(t))| >1/\sqrt{2}.$$
By Lemma \ref{lemma:Jensen:d=2}, it follows that 
$$|\{z \in B(a,N^{7}/\lambda): \Phi(\gamma(z))=0\}| \le M +\log N+O(1)\ll N^7.$$

\section{Checking Assumption \ref{assumption:equi}  for  $\CE_\lambda$ for almost all $\lambda$}\label{section:equi}

Assume
otherwise that for some $r\in \R^{2}$ with $|r| = \frac{1}{2\pi\lambda}$, the set $\{\langle \mu,  r\rangle, \mu \in \CE_\lambda\}$ can be covered by $k=O(N^{\ep_0})$ intervals $I_1,\dots, I_k$ of length $\beta=N^{-1}$ each in $[0,1]$. Consider the disjoint intervals $J_j = (j/3k, (j+1)/3k), 0\le j\le 3k-1$. Let $\ep_0<1$, each interval $I_i, 1\le i\le k$, intersects with at most two intervals $J_{i_1},J_{i_2}$, and so there is one interval $J_{j_0}$ which has no intersection with all $I_1,\dots, I_k$. Thus there is no $\mu\in \CE_\lambda$ such that 
\begin{equation}\label{eqn:J}
\langle \mu,  r\rangle \in J_{j_0}.
\end{equation}
In what follows we just use this simple consequence. Consider $\CE_\lambda$ of $\mu=(\mu_1,\mu_2) \in \Z^2$ such that $\mu_1^2 +\mu_2^2 = m$.
\begin{lemma}\label{lemma:equi:2} For almost all number $m$ up to $x$ that can be written as a sum of two squares, the set $\CE_\lambda$ satisfies Assumption \ref{assumption:equi}.
\end{lemma}

As Assumption \ref{assumption:equi} is on the angles $\alpha_\mu$ of the vectors $(\mu_1,\mu_2) = \sqrt{m} e^{2\pi i \alpha_\mu}$ in $\CE_\lambda$, it suffices to restrict to the set $G(x)$ of $m$ of prime factors congruent with 1 modulo 4 (see \cite{EH}). Indeed, let $D^2$ denote the product of prime factors that are congruent with 3 modulo 4 of $m$, then in any representation of $m$ as $a^2+b^2$, we have $D|a$ and $D|b$, so that $D$ does not affect the angles. Moreover, none of these angles  is influenced by the power of 2 dividing $m$ because if this power is even, the angles are unchanged and if it is odd there is a rotation by $\pi/4$.  We define the discrepancy of the angles $\alpha_\mu$ of the vectors $(\mu_1,\mu_2)$ in $\CE_\lambda$ as follows
$$\Delta_m = \max\Big\{\big | \#\{\alpha_\mu \in [\alpha_1,\alpha_2] \mod 1, \mu\in \CE_\lambda\} - (\alpha_1-\alpha_2) r_2(m)\big |, 0\le \alpha_1\le \alpha_2\le 1 \Big\}.$$
Denote also 
$$R_0(x) = (A +o(1))\frac{x}{\sqrt{\log x}}, A = \frac{1}{2\sqrt{2}} \prod_p (1-\frac{1}{p^2})^{1/2}.$$
Note that $R_0(x)$ is the number of $m\le x$ whose prime divisors are congruent with 1 mod 4 (see again \cite{EH}). Lemma \ref{lemma:equi:2} easily follows from the following result by Erd\H{o}s and Hall. 

\begin{theorem}\cite{EH}\label{theorem:EH} Let $\eps>0$ be fixed. Then for all but $o(R_0(x))$ integers $m \in G(x)$ we have
\begin{equation}\label{eqn:disc:2}
\Delta_m < \frac{r_2(m)}{(\log x)^{\frac{1}{2} \log \frac{\pi}{2} -\eps}}.
\end{equation}
\end{theorem}


We can choose $\ep = .001$ and apply this Theorem to a translation $[\alpha_1, \alpha_2]$ of $J_{j_0}$ to get that the number of $\mu\in \CE_{\lambda}$ with $\langle \mu,  r\rangle \in J_{j_0}$ is at least 
$$N|J_{j_0}|-\frac{r_2(m)}{(\log x)^{\frac{1}{2} \log \frac{\pi}{2} -\eps}} = \frac{N}{3k}-\frac{N}{(\log x)^{\frac{1}{2} \log \frac{\pi}{2} -\eps}}.$$

Since $\sum_{m\le x}r_2(m) = (\pi +o(1)) x$, for almost all $m\in G(x)$ we have $N=r_2(m) \ll \log^{O(1)}(x)$. Thus in this case $k=o\left ((\log x)^{\frac{1}{2} \log \frac{\pi}{2} -\eps}\right )$, and so $J_0$ would contain at least  one point of the set $\{\langle \mu, r\rangle, \mu \in \CE_\lambda\}$, a contradiction.

\section{Proof of Theorem \ref{theorem:subinterval:1&2}}\label{smoothappx}

Under the assumptions of Theorem \ref{theorem:general:d>1}, we deduce Theorem \ref{theorem:subinterval:1&2} from Theorem \ref{theorem:universality1}. The deduction of Theorem \ref{theorem:subinterval:1&2} from Theorem \ref{theorem:universality2} under the setting of Theorem \ref{theorem:general} is completely analogous. 

The task is to pass from smooth test functions to indicator functions. 

Let $l_j = |I_j|=O(1)$. Let $c$ be the constant in Theorem \ref{theorem:real:macro}, and let $\alpha$ be a sufficiently small constant depending on $c$ and $k$. Let $G_j$ be a smooth function that approximates the indicator function $\textbf{1}_{[-l_j/2, l_j/2]}$; in particular, let $G_j$ be supported on $\left [-l_j/2-  N^{-\alpha}, l_j/2+ N^{-\alpha}\right ]$ such that $0\le G_j\le 1$, $G_j = 1$ on $[-l_j/2, l_j/2]$, and $\norm{\triangledown^a G_j} \le CN^{C\alpha}$ for all $0\le a \le 2k$. 

Let $x_j$ be the middle point of $I_j$. We will approximate $\CZ_{j}$ by  
$$\CT_j := \sum G_j(\zeta-x_j)$$
where $\zeta$ runs over all roots of $H$.  

By Theorem \ref{theorem:universality1}, we have
\begin{equation}
\E_{\ep_\mu} \prod_{j=1}^{k} \CT_j - \E_{\Bg} \prod_{j=1}^{k} \CT_j = O\left (N^{-c+C\alpha}\right ) = O\left (N^{-\alpha}\right ) \label{interm1}
\end{equation}
by choosing $\alpha$ sufficiently small.

We will show that for each $j$,
\begin{equation}
\E_{\ep_\mu} |\CT_j-\CZ_j|^{k} = O\left (N^{-\alpha}\right )\label{interm2}
\end{equation}
and for any constant $\alpha'$, 
\begin{equation}
\E_{\ep_\mu} \CT_{j}^{k} = O\left (N^{\alpha'}\right )\label{interm3}.
\end{equation}
Assuming these results, with $\alpha' = \alpha/2k$, by H\"older's inequality and the triangle inequality, we have
$$\E_{\ep_\mu} \prod_{j=1}^{k} \CZ_j - \E_{\ep_\mu} \prod_{j=1}^{k} \CT_j = O\left (N^{-\alpha/k+\alpha'}\right )= O\left (N^{-\alpha/2k}\right ).$$
Combining this with the same bound for the gaussian case and with \eqref{interm1}, we obtain the desired result.

It remains to prove \eqref{interm2} and \eqref{interm3}. The strategy is first to reduce to the Gaussian case using Theorem \ref{theorem:universality1} and then work with the Gaussian case.

Let us prove \eqref{interm3}. By Theorem \ref{theorem:universality1}, we have
\begin{equation}
\E_{\ep_\mu} \CT_{j}^{k} - \E_{\Bg} \CT_{j}^{k} = O\left (N^{-\alpha'}\right )\nonumber.
\end{equation}
Therefore, it suffices to settle the Gaussian case. Note that $\CT_j$ is bounded by $X_j$ defined to be the number of roots of $H$ in the interval $[x_j-l,x_j+ l]$ for $l=l_j/2+N^{-\alpha} = O(1)$. By Jensen's inequality, we have
\begin{eqnarray}
X_j=O(1) \log \frac{K}{|H(x)|}\nonumber 
\end{eqnarray}
where $K = \max_{z\in B(x_j, 2l)}|H(z)|$.  
Thus,
$$\E_{\Bg} X_j^{k} = O(1) \E |\log K|^{k} +O(1)\E |\log |H(x_j)||^{k} .$$
Since $H(x_j)$ is standard gaussian, $ \E |\log |H(x_j)||^{k} = O(1)$. Furthermore, as $|H(x_j)|\le K = O\left (\frac{1}{\sqrt N}\sum_{\mu}|\ep_{\mu, 1}|+|\ep_{\mu, 2}|\right )$, we have
$$ \E |\log |K||^{k} = O(\log^{k} N)$$
proving the desired bound.

Finally, we prove \eqref{interm2}. Since $|\CT_j - \CZ_j|$ is less than the number of roots of $H$ in a union of two intervals of length $N^{-\alpha}$. Approximating the indicator function of each of these intervals by a smooth test function supported on an interval of length $10 N^{-\alpha}$ and applying Theorem \ref{theorem:universality1} to this test function, it suffices to show that for any interval $J = [a, b]$ of length $b-a = O( N^{-\alpha})$, the number of roots of $H$ in $J$, which is denoted by $Y$ satisfies
$$\E_{\Bg} Y^{k} = O(N^{-\alpha}).$$

Assume that it holds for $k=1$. That is $\E_{\Bg} Y = O(N^{-\alpha})$. We have
$$\E_{\Bg} Y^{k} \le O(\E_{\Bg} Y) + \E_{\Bg} Y^{k}\textbf{1}_{Y\ge 2}.$$
By Lemma \ref{repulsionlmm},
$\P_{\Bg} (Y\ge 2)= O\left (N^{-3\alpha/2}\right )$. Since Assumption \eqref{condition:universality} holds true, $Y\le N^{\alpha/k}$ with probability at least $1 - O\left (N^{-A}\right )$ for any constant $A$. Therefore, by condition \eqref{tailex} of Assumption \eqref{condition:universality},
$$\E_{\Bg} \left (Y^{k}\textbf{1}_{Y\ge 2}\right )\le \E_{\Bg} \left (Y^{k}\textbf{1}_{2\le Y\le N^{\alpha/k}}\right ) + \E_{\Bg}\left ( Y^{k}\textbf{1}_{Y\ge N^{\alpha/k}}\right )=O\left (N^{-\alpha/2}\right ).$$

Thus, it remains to prove that $\E_{\Bg} Y = O(N^{-\alpha})$. By the Kac-Rice type formula (see, for instance, \cite[Theorem 2.5]{Far}), one has for every $x\in \R$, 
\begin{eqnarray} 
\E_{\Bg} Y &\le& \int_{a}^{b}\sqrt{\frac{\mathcal S(t)}{\mathcal P(t)^{2}}}dt ,\nonumber
\end{eqnarray}
where $\mathcal P(t)=\Var_{\Bg}(H(t))=1$, $\mathcal Q(t)=\Var_{\Bg}(  H'(t))= \frac{1}{N}\sum_{\mu} \left\langle\mu, \frac{1}{\lambda}\gamma'(t)\right\rangle^{2} = O\left (1\right )$, $\mathcal R(t) = \Cov_{\Bg} (H(t), H'(t))=0$, and $\mathcal S = \mathcal P\mathcal Q - \mathcal R^{2}=\mathcal P\mathcal Q$.
And so, for every $t$,
\begin{eqnarray} 
\frac{\mathcal S(t)}{\mathcal P(t)^{2}} = \frac{\mathcal Q(t)}{\mathcal P(t)} = O\left (1\right )\nonumber
\end{eqnarray}
and 
$$\E_{\Bg} Y =O(1) \int_{a}^{b} 1dt = O\left (N^{-\alpha}\right )$$
as desired.

\section{Proof of Theorems \ref{theorem:general} and \ref{theorem:general:d>1}}\label{section:proofcut}
In this section, we deduce Theorems \ref{theorem:general} and \ref{theorem:general:d>1} from Theorem \ref{theorem:subinterval:1&2}. 
To prove Theorem \ref{theorem:general:d>1}, we partition the interval $[0, \lambda]$ into $\lambda$ intervals $I_1, \dots, I_\lambda$ of length $1$ and apply Theorem \ref{theorem:subinterval:1&2} to every $k$-tuple of these intervals.

To prove Theorem \ref{theorem:general}, we partition the set $\CB_2$ into $M=O(\lambda)$ intervals $I_1, \dots, I_M$ each of length $O(1)$. Applying Theorem \ref{theorem:subinterval:1&2} to every $k$-tuple of these intervals, we get
\begin{equation}
\E_{\ep_\mu} \CZ_{\CB_2}^{k} = \E_{\Bg} \CZ_{\CB_2}^{k} + O(\lambda^{k}/N^{c})\label{interm4}
\end{equation}
where $\CZ_{\CB_2}$ is the number of zeros of $H$ in $\CB_2$.

Let $\CZ' = \CZ - \CZ_{\CB_2}$ be the number of zeros of $H$ in $[0, \lambda]\setminus \CB_2$. 
By \eqref{interm3}, the number of roots $\CZ_j$ of $H$ in each interval $I_j$ satisfies
$$\E_{\ep_\mu} \CZ_j^{h}= O(N^{\alpha})$$ 
for any small constant $\alpha$ and any $h\le k$. 

Thus,  $\E_{\ep_\mu} \CZ_{\CB_2}^{h} = O(\lambda^{h}N^{\alpha})$. By Theorem \ref{theorem:deterministic:d=2}, $\CZ'\ll \lambda N^{-1}$ a.e. Hence, by choosing $\alpha<1-c$
$$\E_{\ep_\mu} \CZ^{k} - \E_{\ep_\mu} \CZ_{\CB_2}^{k} \ll  \lambda^{k} N^{-1+\alpha} \ll \lambda^{k} N^{-c}.$$

This, together with \eqref{interm4}, give the desired result.

\section{Sketch of the proof of Theorem
  \ref{theorem:real:macro}}\label{pgcomplex}

To make the note self-consistent, we present here the main ideas of the
proof; the reader is invited to conslute \cite{OV} for a complete treatment. We first show universality of the complex roots and then deduce Theorem \ref{theorem:real:macro} from it. 
\begin{theorem}[global universality, complex roots]\label{theorem:complex} Let $H(z)=\sum_{\mu} f_\mu(z)$, with $H(z)$ be a random function with $f_{\mu}$ satisfying Assumption \ref{condition:universality}. Let $k$ be an integer constant. For any complex numbers $z_1, \dots, z_k$ in $[0, T]\times [-c, c]$, and for every smooth function $G: \mathbb{C}^{k}\to \mathbb{C}$ supported on $B(0, c)^{k}$ with $|{\triangledown^aG}(z)|\le 1$ for all $0\le a\le 2k+4$ and $z\in \C^{k}$, we have
\begin{equation}
\E_{\xi} \sum_{i_1,\dots,i_k} G (\zeta_{i_1},\dots, \zeta_{i_k}) -\E_{\Bg} \sum_{i_1,\dots,i_k} G (\zeta_{i_1},\dots, \zeta_{i_k}) = O(N^{-c}),\label{gcomplexb}
\end{equation}
where the $\zeta_{i}$ are the roots of $H$, the sums run over all possible assignments of $i_1, \dots, i_k$ which are not necessarily distinct. The constant $c$ here might be different from the constants in Assumption \ref{condition:universality}.
\end{theorem}


\subsection{Sketch of proof of Theorem \ref{theorem:complex}}
By approximation arguments using Fourier expansion, we can reduce the problem to proving \eqref{gcomplexb} for $G$ of the form 
\begin{equation}\label{h2}
 G(w_1,\dots, w_m) = G_1(w_1)\dots G_k(w_k)
 \end{equation} where for each $1\le i\le k$, $G_i:\mathbb{C}\to \mathbb{C}$ is a smooth function supported in $B(0, 1/10)$ and $|{\triangledown^aG_{i}}|\le 1$ for all $0\le a\le 3$.

Let $X_j^{H} = \sum  G_j({\zeta}_i^{H}-z_j)$.
By induction on $k$, it suffices to show that 
\begin{eqnarray}
\ab{\E \prod_{j=1}^{k}X_j^{H}-\E \prod_{j=1}^{k} X_j^{\tilde H}}\le C\delta^{c}.\label{du5}
\end{eqnarray}
Let $A$ be a large constant and $c_1$ be a small positive constant. 
By the Green's formula, one has
\begin{equation}
X_j^{H} =\sum_{i=1}^{n} G_j({\zeta}_i^{H}- z_j)= -\frac{1}{2\pi}\int_{B( z_j, c)}\log |H(z)|\triangle G_j(z- z_j)dz.\label{sat1}
\end{equation}

In the next step, we show that the integral can be approximated by a finite sum with high probability. The technique is based on the Monte-Carlo Lemma, which is in fact a special case of Markov's inequality.
In particular, let $w_{j,1}, \dots, w_{j,m_0}$ be drawn independently at random on the ball $B( z_j, c)$, and let $S$ be the empirical average 
$$S: =  \frac{1}{2c^{2} m_0}\sum_{i=1}^{m_0} \log |H( w_{j,i})|\triangle G_j( w_{j,i}- z_j) .$$
Then by Markov's inequality, we have
\begin{eqnarray}
&&\P\left (\left |S-\frac{1}{2\pi}\int_{B( z_j, c)}\log |H(z)|\triangle G_j(z- z_j)\frac{dz}{\Area(B( z_j, c))}\right |\ge \lambda\right )\nonumber\\
&&\le \frac{1}{m\lambda^{2}} \int_{B( z_j, c)}\left |\log |H(z)|\triangle G_j(z- z_j)\right |^{2}\frac{dz}{\Area(B( z_j, c))} = \frac{O(1)}{m\lambda^{2}} \int_{B( z_j, c)}\left |\log |H(z)|\right |^{2}dz.\nonumber
\end{eqnarray}

Thus, to quantify the approximation of the integral by a finite sum, we need to control the 2-norm of $\log|H|$ on the balls $B(z_j, c)$. That is to bound the function $|H|$ from above and away from 0. These bounds are attained from conditions \eqref{smallballcond} and \eqref{boundedness} of Assumption \eqref{condition:universality}. Note that condition \eqref{smallballcond} only gives a lower bound of $|H|$ for a certain $x\in B( z_j, c)$. To pass from this to a bound that works for all $z\in B( z_j, c)$, one can make use of Harnack's inequality.

Note that on the tail event of conditions \eqref{smallballcond} and \eqref{boundedness}, the approximation is not valid. One has to instead show that the contribution of $X_{j}^{H}$ on that event is negligible. That's when condition \eqref{tailex} becomes handy.

Going back to the good event when we can approximate the integral by a finite sum, we reduce the task of comparing $X_{j}^{H}$ and $X_{j}^{\tilde H}$ to comparing $\sum_{i=1}^{m_0} \log |H( w_{j,i})|\triangle G_j( w_{j,i}- z_j)$ and $\sum_{i=1}^{m_0} \log |\tilde H( w_{j,i})|\triangle G_j( w_{j,i}- z_j)$.  This is done by the Lindeberg swapping argument (see for instance \cite{TV} and the references therein). In particular, by smoothing the $\log$ function, we can further reduce the task to showing that 
for any deterministic $w_{j,i}$ with $1\le j\le k$, $1\le i\le m_0$, and for a smooth function $L:\mathbb{C}^{km_0}\to \mathbb{C}$, 
\[\left|\E L\left( H( w_{j,i})\right)_{ji} -\E L\left(\tilde H( w_{j,i}) \right)_{ji} \right|\le CN^{-c}.
 \]
The swapping method uses the triangle inequality to bound the above difference by a sum of $2N$ differences each of which involves changing only one random variable to gaussian. For example, one of these differences is 
$\E L\left( H_0( w_{j,i})\right)_{ji} -\E L\left(H_1( w_{j,i})\right )$ where $H_0(z) = H(z) = \sum_{\mu} \xi_{\mu}f_{\mu}(z)$ and $H_1(z) = \tilde \xi_{\mu_1}f_{\mu_1}(z) + \sum_{\mu\neq \mu_1} \xi_{\mu}f_{\mu}(z)$. We then Taylor expand the function $L\left( H_0( w_{j,i})\right)_{ji}$ (and $L\left( H_1( w_{j,i})\right)_{ji}$) as a function of one variable $\xi_{\mu}$ (and $\tilde \xi_{\mu}$ respectively). Making use of the assumption that the first and second moments of $\xi_{\mu}$ and $\tilde \xi_{\mu}$ are the same, one can see that upon taking expectation, the first three terms in the Taylor expansions cancel out, leaving us with a small error term. Adding up these errors terms, one obtains $N^{-c}$ as desired. The reader may notice that this is quite similar to a classical proof of the Central Limit Theorem using the swapping argument.

\subsection{Universality of real roots: sketch of proof of Theorem \ref{theorem:real:macro}} \label{pgreal}

As in the proof of Theorem \ref{theorem:complex}, we can reduce the problem to showing that
\begin{eqnarray}
\ab{\E \left(\prod_{j=1}^{k}X_{ {x}_i, G_i, \mathbb{R}}^{H}\right) -\E \left(\prod_{j=1}^{k}X_{ {x}_i, G_i, \mathbb{R}}^{\tilde H}\right)} \le C' N^{-c},\label{du6}
\end{eqnarray}
where $X^{H}_{ {x}_i, G_i, \mathbb{R}} = \sum_{\zeta_j^{H}\in\mathbb{R}}G_i(\zeta_j^{ H}- x_i)$, $\zeta_j^{H}$ are the roots of $H$, and $H_i:\mathbb{R}\to\mathbb{C}$ are smooth functions supported on $[-c, c]$ and $B(0, c)$ respectively, such that
 $|{\triangledown^{a}G_i}(x)| \le 1
$
for all $1\le i\le k $, $x\in \R$,  and $0\le a\le 3$.

The idea is to reduce it to Theorem \ref{theorem:complex}. This is done by showing that the number of complex zeros near the real axis is small with high probability.
\begin{lemma}\label{repulsionlmm} We have
$$\P \left(\CZ_{H}{B( x,\gamma)}\ge 2\right) \le C\gamma^{3/2},
\qquad\text{for all } x\in [0, T] $$
where $\gamma=N^{-c}$ for any sufficiently small constant $c$.
\end{lemma}

Using Theorem \ref{theorem:complex}, this lemma is reduced to the Gaussian case. Let $\tilde H(z) = \sum_{\mu} \tilde \xi_{\mu} f_{\mu}(z)$ where $\tilde \xi_{\mu}$ are standard gaussian. Let $g(z) = \tilde H(x) + \tilde H'(x)(z-x)$ and $p(z) = \tilde H(z) - g(z)$. By Rouch\'{e}'s theorem, 
 $$\P(\CZ_{\tilde{H}}{ B(x,2\gamma)}\ge 2) \le \P\left (\min_{z\in \partial B(x, 2\gamma)}|g(z)|\le \max_{z\in\partial B(x, 2\gamma)}|p(z)|\right ).$$

Both $g(z)$ and $p(z)$ have zero mean. Condition \eqref{secondDev} of Assumption \eqref{condition:universality} shows that for all $z\in B(x, 2\gamma)$,
\begin{equation}
\Var (p(z))= O\left ( N^{-(4+\ep)c}  \Var (\tilde H (x))\right).\nonumber
\end{equation}

Thus with probability at least $1 - O\left (N^{-3c/2}\right)$, 
\begin{equation}
\max_{z\in \partial B(x, 2\gamma)} |p(z)|=O\left ( N^{-(2+\ep)c}  \sqrt{\Var (\tilde H (x))}\right).\label{pbound}
\end{equation}

Now, for $g$, note that since $g$ is a linear function with real coefficients, one has $\min_{z\in \partial B(x, 2\gamma)} |g(z)| = \min |g(x \pm 2\gamma)|$. Condition \ref{firstDev} shows that $g(x\pm 2\gamma)$ is normally distributed with variance
\begin{equation}
\Var(g(x \pm 2\gamma)) \ge 1/2 \Var (\tilde H (x)) \nonumber.
\end{equation}
Therefore, with probability at least $1 - O\left (N^{-3c/2}\right)$, 
$$|g(x \pm 2\gamma)|\ge N^{-3c/2}  \sqrt{\Var (\tilde H (x))}$$
Combining this with \eqref{pbound}, we obtain Lemma \ref{repulsionlmm}.

{\bf Acknowledgement.} The authors are grateful to prof. Z. Rudnick for  helpful comments.

\appendix

\end{document}